\documentclass{amsart}
\usepackage{amsmath}
  \usepackage{paralist}
  \usepackage{amssymb}
 \usepackage{amsthm}
\usepackage[colorlinks=true]{hyperref}
\hypersetup{urlcolor=blue, citecolor=red}

\newtheorem{theorem}{Theorem}[section]
\newtheorem{corollary}[theorem]{Corollary}
\newtheorem{lemma}[theorem]{Lemma}

\theoremstyle{definition}
\newtheorem{definition}[theorem]{Definition}
\newtheorem{remark}[theorem]{Remark}

\numberwithin{equation}{section}

\title[Multiplicity and orbital stability]
      {Multiplicity and orbital stability of  normalized solutions to non-autonomous  Schr\"{o}dinger equation with mixed nonlinearities}

\author[Xinfu Li, Li Xu, Meiling Zhu]{}


\thanks{*Corresponding author. Email Addresses:  lxylxf@tjcu.edu.cn; beifang\_xl@163.com.}

\begin{document}
\maketitle

\centerline{\scshape Xinfu Li$^a$,\  Li Xu$^{a*}$,\ Meiling Zhu$^b$}
\medskip
{\footnotesize
 \centerline{$^a$School of Science, Tianjin University of Commerce, Tianjin 300134,
P. R. China}\centerline{$^\mathrm{b}$College of Computer Science and
Engineering, Cangzhou Normal University,} \centerline{Cangzhou,
Hebei, 061000, P. R. China}}

\bigskip

\begin{abstract}
This paper studies the  multiplicity  of normalized solutions to the
Schr\"{o}dinger equation with mixed nonlinearities
\begin{equation*}
\begin{cases}
-\Delta u=\lambda u+h(\epsilon x)|u|^{q-2}u+\eta |u|^{p-2}u,\quad
x\in \mathbb{R}^N, \\
\int_{\mathbb{R}^N}|u|^2dx=a^2,
\end{cases}
\end{equation*}
where  $a, \epsilon, \eta>0$, $q$ is $L^2$-subcritical, $p$ is
$L^2$-supercritical, $\lambda\in \mathbb{R}$ is an unknown parameter
that appears as a Lagrange multiplier, $h$ is a positive and
continuous function. It is proved that the numbers of normalized
solutions are at least the numbers of global maximum points of $h$
when $\epsilon$ is small enough. Moreover, the orbital stability of
the solutions obtained is analyzed as
well. In particular, our results cover the Sobolev critical case $p=2N/(N-2)$.
\end{abstract}

\section{Introduction and main results}

\setcounter{section}{1}
\setcounter{equation}{0}

In this paper, we study the multiplicity and orbital stability of
normalized solutions to the non-autonomous Schr\"{o}dinger equation
with mixed nonlinearities
\begin{equation}\label{e1.1}
\begin{cases}
-\Delta u=\lambda u+h(\epsilon x)|u|^{q-2}u+\eta |u|^{p-2}u,\quad
x\in \mathbb{R}^N, \\
\int_{\mathbb{R}^N}|u|^2dx=a^2,
\end{cases}
\end{equation}
where  $N\geq 1,  a, \epsilon, \eta>0$, $2<q<2+\frac{4}{N}<p\left\{
\begin{array}{ll}
< +\infty, & N=1,2,\\
\leq 2^*:= \frac{2N}{N-2}, &N\geq 3,
\end{array}
\right.$ and $\lambda\in \mathbb{R}$ is an unknown parameter that
appears as a Lagrange multiplier. The function $h$  satisfies the
following conditions:\\
($h_1$) $h\in C(\mathbb{R}^N,\mathbb{R})$ and $0<h_0=\inf_{x\in
\mathbb{R}^N}h(x)\leq \max_{x\in \mathbb{R}^N}=h_{\text{max}}$;\\
($h_2$) $h_{\infty}=\lim_{|x|\to +\infty}h(x)<h_{\text{max}}$;\\
($h_3$) $h^{-1}(h_{\text{max}})=\{a_1,a_2,\cdots,a_l\}$ with $a_1=0$
and $a_j\neq a_i$ if $i\neq j$.

A solution $u$ to the problem (\ref{e1.1}) corresponds to a critical
point of the functional
\begin{equation}\label{e1.2}
E_{\epsilon}(u):=\frac{1}{2}\int_{\mathbb{R}^N}|\nabla
u|^2dx-\frac{1}{q}\int_{\mathbb{R}^N}h(\epsilon
x)|u|^{q}dx-\frac{\eta}{p}\int_{\mathbb{R}^N}|u|^{p}dx
\end{equation}
restricted to the sphere
\begin{equation*}
S(a):=\{u\in H^1(\mathbb{R}^N):\int_{\mathbb{R}^N}|u|^2dx=a^2\}.
\end{equation*}
It is well known that $E_{\epsilon}\in
C^1(H^1(\mathbb{R}^N),\mathbb{R})$ and
\begin{equation*}
E_{\epsilon}'(u)\varphi=\int_{\mathbb{R}^N}\nabla u\nabla \varphi
dx-\int_{\mathbb{R}^N}h(\epsilon x)|u|^{q-2}u\varphi
dx-\eta\int_{\mathbb{R}^N}|u|^{p-2}u\varphi dx
\end{equation*}
for any $\varphi\in H^1(\mathbb{R}^N)$.

One motivation driving the search for normalized solutions to
(\ref{e1.1}) is the nonlinear Schr\"{o}dinger equation
\begin{equation}\label{e1.3}
i\frac{\partial \psi}{\partial t}+\Delta \psi+g(|\psi|^2)\psi=0,\
(t,x)\in \mathbb{R}\times\mathbb{R}^N.
\end{equation}
Since the mass $\int_{\mathbb{R}^N}|\psi|^2dx$ is preserved along
the flow associated with  (\ref{e1.3}), it is natural to consider it
as prescribed. Searching for standing wave solution
$\psi(t,x)=e^{-i\lambda t}u(x)$ of (\ref{e1.3})  leads to
(\ref{e1.1}) for $u$ with $g(|s|^2)s = h(\epsilon x)|s|^{q-2}s+\eta
|s|^{p-2}s$. In recent decades, the research of finding normalized
solutions to Schr\"{o}dinger equations has received a special
attention. This seems to be particularly meaningful from the
physical point of view, as the $L^2$-norm is a preserved quantity of
the evolution and the variational characterization of such solutions
is often a strong help to analyze their orbital stability, see
\cite{{Bellazzini,Jeanjean,Luo13},{Cazenave-Lions82},{SoaveJDE},{SoaveJFA}}
and the references therein.

In the study of normalized solutions to the Schr\"{o}dinger equation
\begin{equation}\label{e1.4}
-\Delta u=\lambda u+|u|^{p-2}u,\quad x\in \mathbb{R}^N,
\end{equation}
the number $\bar{p}:= 2 + 4/ N$, labeled $L^2$-critical exponent, is
a very important number, because in the study of (\ref{e1.4}) using
variational methods, the functional
\begin{equation*}
J(u):=\frac{1}{2}\int_{\mathbb{R}^N}|\nabla
u|^2dx-\frac{1}{p}\int_{\mathbb{R}^N}|u|^pdx,\ u\in
H^1(\mathbb{R}^N)
\end{equation*}
is bounded from below on $S(a)$ for the $L^2$-subcritical problem,
i.e., $2 < p < 2 + 4/ N$. Thus, a solution of (\ref{e1.4}) can be
found as a global minimizer of $J|_{S(a)}$, see
\cite{{Shibata17},{Stuart82}}. For the purely $L^2$-supercritical
problem, i.e., $2+ 4/N < p < 2^*$, $J|_{S(a)}$ is unbounded from
below (and from above). Related to this case, a seminar paper due to
Jeanjean \cite{Jeanjean97} exploited the mountain pass geometry to
get a normalized solution, see
\cite{{Bartsch-Valeriola13},{Bartsch-Soave17},{Bieganowski-Mederski20},{Hirata-Tanaka19},{Ikoma-Tanaka19},{Jeanjean-Lu19},{Jeanjean-Lu
20}} for more results. In the purely $L^2$-critical case (i.e., $p =
2+ 4/N$), the result is delicate. Recently, the Schr\"{o}dinger
equation with double power form nonlinearity
$\mu|u|^{q-2}u+|u|^{p-2}u$ has been extensively studied due to Soave
\cite{{SoaveJDE},{SoaveJFA}}, see
\cite{{Alves-Chao21},{JEANJEAN-JENDREJ},{Jeanjean-Le-MA},{Li-Cv-21},{Stefanov
19},{Wei-Wu21}} for more results. The multiplicity of normalized
solutions to the autonomous Schr\"{o}dinger equation or systems  has
also been considered extensively at the last years, see
\cite{{Alves-Chao21},{Bartsch-Valeriola13},{Bartsch-Soave19},{Gou-Jeanjean18},
{Hirata-Tanaka19},{Ikoma-Tanaka19},{Jeanjean-Le-MA},{Jeanjean-Lu19},{Jeanjean-Lu
20},{Luo-Wei-Yang-Zhen22},{Luo-Yang-Zou21}}. As for the existence of
normalized solutions to  the non-autonomous Schr\"{o}dinger equation
\begin{equation}\label{e1.5}
-\Delta u=\lambda u+f(x,u),\quad x\in \mathbb{R}^N,
\end{equation}
we refer to
\cite{{Bellazzini17},{Ikoma-Miyamoto20},{Li-Zhao20},{Luo-Peng19},{Pellacci-Pistoia19}}
and the references therein.

Our study is motivated by Alves \cite{Alves-ZAMP22}, where they
considered the multiplicity of normalized solutions to
\begin{equation}\label{e1.6}
-\Delta u=\lambda u+h(\epsilon x)f(u),\quad x\in \mathbb{R}^N
\end{equation}
with $f$ being $L^2$-subcritical. Their arguments depend on the
existence of global minimizer and the relative compactness of any
minimizing sequence of the functional $\tilde{J}|_{S(a)}$
corresponding to the limit problem
\begin{equation}\label{e1.7}
-\Delta u=\lambda u+\mu f(u),\quad x\in \mathbb{R}^N.
\end{equation}
While in our problem  (\ref{e1.1}), the appearance of the
$L^2$-supercritical term $\eta|u|^{p-2}u$ makes the functional to
the limit problem
\begin{equation}\label{e1.8}
-\Delta u=\lambda u+\mu|u|^{q-2}u+\eta|u|^{p-2}u,\quad x\in
\mathbb{R}^N
\end{equation}
with $q<2+4/N<p$ is unbounded from below (and from above). But in
view of the studies of \cite{{Jeanjean-Le-MA},{SoaveJDE}}, we know
that the functional in this case has a local minimizer. So employing
the truncated skill used in \cite{{Alves-Chao21},{Peral Alonso
1997}}, we can isolate  the local minimizer and obtain the
multiplicity of normalized solutions to the problem (\ref{e1.1}).
The application of truncated functions and the appearance of the
Sobolev critical exponent $p=2^*$ make more delicate analysis is
needed. Furthermore, we also consider the orbital stability of the
solutions obtained (see Section 5). We should point out that in
\cite{Alves-Chao21}, the authors studied the multiplicity of
normalized solutions to the autonomous Schr\"{o}dinger equation
(\ref{e1.8}) with $q<2+4/N<p=2^*$ in radial symmetry space
$H_{rad}^1(\mathbb{R}^N)$ by using truncated skill and genus theory.
Note that our problem (\ref{e1.1}) is non-autonomous and not
radially symmetry, so their method is not work in our problem.

The main results of this paper are as follows.

\begin{theorem}\label{thm1.1}
Let $N,\epsilon,a,\eta,p,q,h$ be as in (\ref{e1.1}).
 We further assume that $h_{\text{max}}, a$ and $\eta$ satisfy some condition (i.e.,(\ref{e2.1}) holds
for $p<2^*$, (\ref{e2.1}) and (\ref{e2.4}) hold for $p=2^*$). Then
there exists $\epsilon_0>0$ such that (\ref{e1.1}) admits at least
$l$ couples $(u_j,\lambda_j)\in H^1(\mathbb{R}^N)\times \mathbb{R}$
of weak solutions for $\epsilon\in (0,\epsilon_0)$ with
$\int_{\mathbb{R}^N}|u_j|^2dx=a^2$, $\lambda_j<0$ and
$E_{\epsilon}(u_j)<0$ for $j=1,2,\cdots,l$.
\end{theorem}

\begin{theorem}\label{rmk1.2}
The solutions obtained in Theorem \ref{thm1.1} is orbitally stable
in some sense. To state this theorem,  we need  some notations used
in the proof of Theorem \ref{thm1.1}, so we give the details in
Section 5. We point out that we give the stability of $l$ different
sets, which is very different from the existing results.
\end{theorem}

\begin{remark}\label{rmk1.3}
In \cite{{JEANJEAN-JENDREJ},{Jeanjean-Le-MA},{SoaveJDE},{SoaveJFA}},
the authors considered the normalized solutions to (\ref{e1.8}) with
$q<2+4/N<p$. They obtained a ground state solution to (\ref{e1.8})
with negative energy which is local minimizer and orbitally stable
and a mountain-pass type solution with positive energy which is
strongly instable. In this paper, the solutions obtained in Theorem
\ref{thm1.1} are also local minimizers, but we do not know whether
they are ground state solutions. The appearance of the potential $h$
increases the number of the local minimizer and maintains its
stability.
\end{remark}

This paper is organized as follows. In Section 2, we define the
truncated functional used in the study. In Section 3, we study the
properties of the truncated autonomous functional. In Section 4, we
study the truncated non-autonomous problem and give the proof of
Theorem \ref{thm1.1}. In Section 5, we study the orbital stability
of the solutions obtained in Theorem \ref{thm1.1}.

\medskip

\textbf{Notation}:  For $t\geq 1$, the  $L^t$-norm of $u\in
L^t(\mathbb{R}^N)$ is denoted by $\|u\|_t$. The usual norm of $u\in
H^1(\mathbb{R}^N)$ is denoted by $\|u\|$. $C,C_1,C_2,\cdots$ denote
any positive constant, whose value is not relevant and may be change
from line to line. $o_n(1)$ denotes a real sequence with $o_n(1) \to
0$ as $n\to +\infty$. `$\rightarrow$' denotes  strong convergence
and `$\rightharpoonup$' denotes weak convergence.
$B_{r}(x_0):=\{x\in\mathbb{R}^N:|x-x_0|<r\}$.

\section{Truncated functionals}
\setcounter{section}{2} \setcounter{equation}{0}

In the proof of Theorem \ref{thm1.1}, we will adapt for our case a
truncated function found in Peral Alonso (\cite{Peral Alonso 1997},
Chapter 2, Theorem 2.4.6).

In what follows, we will consider the functional $E_\epsilon$ given
by (\ref{e1.2}) restricts to $S(a)$. By the Sobolev embedding and
the Gagliardo-Nirenberg inequality (see \cite{Weinstein 1983})
\begin{equation}\label{e2.9}
\begin{cases}
\|u\|_{t}\leq C_{N,t}\|u\|_2^{1-\gamma_t}\|\nabla u\|_2^{\gamma_t},\
2<t\left\{
\begin{array}{ll}
<\infty, & N=1,2,\\
\leq 2^*, &N\geq 3,
\end{array}
\right.   \gamma_t:=\frac{N}{2}-\frac{N}{t},\\
 C_{N,2^*}=S^{-\frac{1}{2}},\
\text{and}\  S:=\inf_{ u\in
D^{1,2}(\mathbb{R}^N)\setminus\{0\}}\frac{\int_{\mathbb{R}^N}|\nabla
u|^2dx}{\left(\int_{\mathbb{R}^N}|u|^{2^*}dx\right)^{2/2^*}},
\end{cases}
\end{equation}
we have
\begin{equation}\label{e2.13}
\begin{split}
E_{\epsilon}(u)&\geq  \frac{1}{2}\int_{\mathbb{R}^N}|\nabla
u|^2dx-\frac{1}{q}h_{\text{max}}\int_{\mathbb{R}^N}|u|^{q}dx-\frac{\eta}{p}\int_{\mathbb{R}^N}|u|^{p}dx\\
&\geq \frac{1}{2}\|\nabla
u\|_2^2-\frac{1}{q}h_{\text{max}}C_{N,q}^qa^{q(1-\gamma_q)}\|\nabla
u\|_2^{q\gamma_q}-\frac{\eta}{p}C_{N,p}^pa^{p(1-\gamma_p)}\|\nabla
u\|_2^{p\gamma_p}\\
&=g_a(\|\nabla u\|_2)
\end{split}
\end{equation}
for any $u\in S(a)$, where
\begin{equation*}
g_a(r):=\frac{1}{2}r^2-\frac{1}{q}h_{\text{max}}C_{N,q}^qa^{q(1-\gamma_q)}r^{q\gamma_q}
-\frac{\eta}{p}C_{N,p}^pa^{p(1-\gamma_p)}r^{p\gamma_p},\ r>0.
\end{equation*}
Set $g_a(r)=r^2w_a(r)$ with
\begin{equation*}
w_a(r):=\frac{1}{2}-\frac{1}{q}h_{\text{max}}C_{N,q}^qa^{q(1-\gamma_q)}r^{q\gamma_q-2}
-\frac{\eta}{p}C_{N,p}^pa^{p(1-\gamma_p)}r^{p\gamma_p-2},\ r>0.
\end{equation*}

Now we study the properties of $w_a(r)$. Note that
\begin{equation*}
t\gamma_t\left\{
\begin{array}{ll}
<2, & 2<t<2+4/N,\\
=2, & t=2+4/N,\\
>2, & 2+4/N<t\leq 2^*
\end{array}
\right.\quad\text{and}\quad \gamma_{2^*}=1.
\end{equation*}
It is obvious that $\lim_{r\to 0^+}w_a(r)=-\infty$ and  $\lim_{r\to
+\infty}w_a(r)=-\infty$. By direct calculations, we obtain that
\begin{equation*}
w_a'(r)=-\frac{1}{q}h_{\text{max}}C_{N,q}^qa^{q(1-\gamma_q)}(q\gamma_q-2)r^{q\gamma_q-3}
-\frac{\eta}{p}C_{N,p}^pa^{p(1-\gamma_p)}(p\gamma_p-2)r^{p\gamma_p-3}.
\end{equation*}
Then the equation $w_a'(r)=0$ has a unique solution
\begin{equation*}
r_0=\left(\frac{(2-q\gamma_q)\frac{1}{q}h_{\text{max}}C_{N,q}^qa^{q(1-\gamma_q)}}
{(p\gamma_p-2)\frac{\eta}{p}C_{N,p}^pa^{p(1-\gamma_p)}}\right)^{\frac{1}{p\gamma_p-q\gamma_q}}
\end{equation*}
and the maximum of $w_a(r)$ is
\begin{equation*}
w_a(r_0)=\frac{1}{2}-B\left(h_{\text{max}}a^{q(1-\gamma_q)}\right)^{\frac{p\gamma_p-2}{p\gamma_p-q\gamma_q}}
\left(\eta
a^{p(1-\gamma_p)}\right)^{\frac{2-q\gamma_q}{p\gamma_p-q\gamma_q}},
\end{equation*}
where
\begin{equation*}
B=\frac{p\gamma_p-q\gamma_q}{2-q\gamma_q}\left(\frac{2-q\gamma_q}{p\gamma_p-2}\right)^{\frac{p\gamma_p-2}{p\gamma_p-q\gamma_q}}
\left(\frac{C_{N,q}^q}{q}\right)^{\frac{p\gamma_p-2}{p\gamma_p-q\gamma_q}}
\left(\frac{C_{N,p}^p}{p}\right)^{\frac{2-q\gamma_q}{p\gamma_p-q\gamma_q}}.
\end{equation*}
Thus, if we assume
\begin{equation}\label{e2.1}
\left(h_{\text{max}}a^{q(1-\gamma_q)}\right)^{\frac{p\gamma_p-2}{p\gamma_p-q\gamma_q}}
\left(\eta
a^{p(1-\gamma_p)}\right)^{\frac{2-q\gamma_q}{p\gamma_p-q\gamma_q}}<\frac{1}{2B},
\end{equation}
then the maximum of $w_a(r)$ is positive and $w_a(r)$ has exactly
two zeros $0<R_0<R_1<\infty$, which are also the zeros of $g_a(r)$.
It is obvious that $g_a(r)$ has the following properties
\begin{equation}\label{e2.11}
\begin{cases}
&g_a(0)=g_a(R_0)=g_a(R_1)=0;\ g_a(r)<0\ \text{for}\ r>0\
\text{small};\\
&\lim_{r\to+\infty}g_a(r)=-\infty;\  g_a(r)\ \text{has\ exactly\
two\
critical\ points\ }\\
&r_1\in (0,R_0)\ \text{and}\ r_2\in (R_0,R_1)\ \text{with}\
g_a(r_1)<0 \ \text{and}\ g_a(r_2)>0.
\end{cases}
\end{equation}
For $p=2^*$, we further assume that
$R_0<\eta^{-\frac{N-2}{4}}S^{\frac{N}{4}}$, which is satisfied if we
assume
\begin{equation}\label{e2.3}
r_0<\eta^{-\frac{N-2}{4}}S^{\frac{N}{4}},
\end{equation}
because $R_0<r_0<R_1$. By the expression of $r_0$, (\ref{e2.3}) is
equivalent to
\begin{equation}\label{e2.4}
\begin{split}
\left(h_{\text{max}}a^{q(1-\gamma_q)}\right)^{\frac{1}{p\gamma_p-q\gamma_q}}
&\eta^{\frac{N-2}{4}-\frac{1}{p\gamma_p-q\gamma_q}}\\
&\leq \left(\frac{(2-q\gamma_q)C_{N,q}^q2^*S^{\frac{2^*}{2}}}
{q(p-2)}\right)^{\frac{-1}{p\gamma_p-q\gamma_q}}S^{\frac{N}{4}}.
\end{split}
\end{equation}

Now fix $\tau: (0,+\infty)\to [0,1]$ as being a non-increasing and
$C^\infty$ function that satisfies
\begin{equation}\label{e2.7}
\tau(x)=\left\{
\begin{array}{ll}
1, &\text{if}\ x\leq R_0,\\
0, &\text{if}\ x\geq R_1
\end{array}
\right.
\end{equation}
and  consider the truncated functional
\begin{equation}\label{e2.2}
E_{\epsilon,T}(u):=\frac{1}{2}\int_{\mathbb{R}^N}|\nabla
u|^2dx-\frac{1}{q}\int_{\mathbb{R}^N}h(\epsilon
x)|u|^{q}dx-\frac{\eta}{p}\tau(\|\nabla
u\|_2)\int_{\mathbb{R}^N}|u|^{p}dx.
\end{equation}
Thus,
\begin{equation}\label{e2.8}
\begin{split}
E_{\epsilon,T}(u)&\geq \frac{1}{2}\|\nabla
u\|_2^2-\frac{1}{q}h_{\text{max}}C_{N,q}^qa^{q(1-\gamma_q)}\|\nabla
u\|_2^{q\gamma_q}\\
&\qquad-\frac{\eta}{p}\tau(\|\nabla
u\|_2)C_{N,p}^pa^{p(1-\gamma_p)}\|\nabla
u\|_2^{p\gamma_p}\\
&=\bar{g}_a(\|\nabla u\|_2)
\end{split}
\end{equation}
for any $u\in S(a)$, where
\begin{equation*}
\bar{g}_a(r):=\frac{1}{2}r^2-\frac{1}{q}h_{\text{max}}C_{N,q}^qa^{q(1-\gamma_q)}r^{q\gamma_q}
-\frac{\eta}{p}\tau(r)C_{N,p}^pa^{p(1-\gamma_p)}r^{p\gamma_p}.
\end{equation*}
It is easy to see  that $\bar{g}_a(r)$ has the following properties
\begin{equation}\label{e2.12}
\begin{cases}
&\bar{g}_a(r)\equiv {g}_a(r)\ \text{for\ all}\ r\in [0,R_0];\\
&\bar{g}_a(r)\ \text{is\ positive\ and\ strictly\ increasing\ in\
}(R_0,+\infty).
\end{cases}
\end{equation}

Correspondingly, for any $\mu\in (0,h_{\text{max}}]$, we denote by
$J_{\mu},\ J_{\mu,T}: H^1(\mathbb{R}^N)\to\mathbb{R}$ the following
functionals
\begin{equation}\label{e2.5}
J_{\mu}(u):=\frac{1}{2}\int_{\mathbb{R}^N}|\nabla
u|^2dx-\frac{\mu}{q}\int_{\mathbb{R}^N}|u|^{q}dx-\frac{\eta}{p}\int_{\mathbb{R}^N}|u|^{p}dx
\end{equation}
and
\begin{equation}\label{e2.6}
J_{\mu,T}(u):=\frac{1}{2}\int_{\mathbb{R}^N}|\nabla
u|^2dx-\frac{\mu}{q}\int_{\mathbb{R}^N}|u|^{q}dx-\frac{\eta}{p}\tau(\|\nabla
u\|_2)\int_{\mathbb{R}^N}|u|^{p}dx.
\end{equation}

The properties of $J_{\mu,T}$ and $E_{\epsilon,T}$ will be studied
in Sections 3 and 4, respectively.

\section{The truncated autonomous functional}
\setcounter{section}{3} \setcounter{equation}{0}

In this section, we study the properties of the functional
$J_{\mu,T}$ defined in (\ref{e2.6}) restricted to $S(a_1)$, where
$\mu\in(0,h_{\text{max}}]$ and $a_1\in (0,a]$.

\begin{lemma}\label{lem3.2}
Let $N, a,\eta,p,q$ be as in (\ref{e1.1}), (\ref{e2.1}) hold,
$\mu\in (0,h_{\text{max}}]$, $0<a_1\leq a$.
 Then the functional $J_{\mu,T}$ is bounded from below in $S(a_1)$.
\end{lemma}

\begin{proof}
By (\ref{e2.8}) and (\ref{e2.12}), for any $u\in S(a_1)$,
\begin{equation*}
\begin{split}
J_{\mu,T}(u)&\geq J_{\text{max},T}(u) \\
& \geq \frac{1}{2}\|\nabla
u\|_2^2-\frac{1}{q}h_{\text{max}}C_{N,q}^qa_1^{q(1-\gamma_q)}\|\nabla
u\|_2^{q\gamma_q}\\
&\qquad\qquad-\frac{\eta}{p}\tau(\|\nabla
u\|_2)C_{N,p}^pa_1^{p(1-\gamma_p)}\|\nabla u\|_2^{p\gamma_p}\\
&\geq \bar{g}_a(\|\nabla u\|_2)\geq \inf_{r\geq
0}\bar{g}_a(r)>-\infty.
\end{split}
\end{equation*}
The proof is complete.
\end{proof}

\begin{lemma}\label{lem3.4}
Let $N, a,\eta,p,q$ be as in (\ref{e1.1}), (\ref{e2.1}) hold,
$\mu\in (0,h_{\text{max}}]$, $0<a_1\leq a$.
$\Upsilon_{\mu,T,a_1}:=\inf_{u\in S(a_1)}J_{\mu,T}(u)<0$.
\end{lemma}

\begin{proof}
Fix $u\in S(a_1)$. For $t>0$, we define
$u_t(x)=t^{\frac{N}{2}}u(tx)$. Then $u_t\in S(a_1)$ for all $t>0$.
By $\tau\geq 0$ and $q\gamma_q<2$, we obtain that
\begin{equation*}
\begin{split}
J_{\mu,T}(u_t)&\leq \frac{1}{2}\int_{\mathbb{R}^N}|\nabla
u_t|^2dx-\frac{\mu}{q}\int_{\mathbb{R}^N}|u_t|^{q}dx\\
&=\frac{1}{2}t^2\int_{\mathbb{R}^N}|\nabla
u|^2dx-\frac{\mu}{q}t^{q\gamma_q}\int_{\mathbb{R}^N}|u|^{q}dx\\
&<0
\end{split}
\end{equation*}
for $t>0$ small enough. Thus $\Upsilon_{\mu,T,a_1}<0$. The proof is
complete.
\end{proof}

\begin{lemma}\label{lem3.1}
Let $N, a,\eta,p,q$ be as in (\ref{e1.1}), (\ref{e2.1}) hold,
$\mu\in (0,h_{\text{max}}]$. Then\\
(1) $J_{\mu,T}\in C^1(H^1(\mathbb{R}^N),\mathbb{R})$;\\
(2) Let
$a_1\in (0,a]$. If $u\in S(a_1)$ such that $J_{\mu,T}(u)<0$, then
$\|\nabla u\|_2<R_0$ and $J_{\mu,T}(v) = J_{\mu} (v)$ for all $v$
satisfying $\|v\|_2\leq a$ and being in a small neighborhood of $u$
in $H^1(\mathbb{R}^N)$.
\end{lemma}

\begin{proof}
(1) is trivial. Now we prove (2). It follows from $J_{\mu,T}(u)<0$
and $$J_{\mu,T}(u)\geq \bar{g}_{a_1}(\|\nabla u\|_2)\geq
\bar{g}_{a}(\|\nabla u\|_2)$$ that $\bar{g}_a(\|\nabla u\|_2)<0$,
which implies that $\|\nabla u\|_2<R_0$ by (\ref{e2.12}). By (1) and
$J_{\mu,T}(u)<0$, we obtain that $J_{\mu,T}(v)<0$ for all $v$ in a
small neighborhood of $u$ in $H^1 (\mathbb{R}^N)$, which combined
with $\|v\|_2\leq a$ gives that  $\|\nabla v\|_2<R_0$ and thus
$J_{\mu,T}(v) = J_{\mu} (v)$. The proof is complete.
\end{proof}

For any $a_1\in (0,a]$, we define
\begin{equation*}
m_{\mu}(a_1):=\inf_{u\in V(a_1)}J_\mu(u),\ V(a_1):=\{u\in
S(a_1):\|\nabla u\|_2<R_0\}.
\end{equation*}
Since $J_{\mu,T}(u)\geq \bar{g}_{a_1}(\|\nabla u\|_2)\geq
\bar{g}_{a}(\|\nabla u\|_2)$ for any $u\in S(a_1)$, by Lemmas
\ref{lem3.2}-\ref{lem3.1}, we obtain that
\begin{equation}\label{e3.2}
\Upsilon_{\mu,T,a_1}=\inf_{u\in S(a_1)}J_{\mu,T}(u)=m_{\mu}(a_1).
\end{equation}
In (\cite{JEANJEAN-JENDREJ}, Lemma 2.6 and Theorem 1.2), the authors
obtained that

\begin{lemma}\label{lem3.3}
Let $N, a,\eta,p,q$ be as in (\ref{e1.1}), (\ref{e2.1}) hold, $\mu\in (0,h_{\text{max}}]$. Then \\
(1) $a_1\in (0,a]\mapsto m_{\mu}(a_1)$ is continuous;\\
(2) Let $0<a_1<a_2\leq a$, then
$\frac{a_1^2}{a_2^2}m_{\mu}(a_2)<m_{\mu}(a_1)<0$.
\end{lemma}

Consequently, by (\ref{e3.2}) and Lemma \ref{lem3.3}, we obtain that

\begin{lemma}\label{lem3.5}
Let $N, a,\eta,p,q$ be as in (\ref{e1.1}), (\ref{e2.1}) hold, $\mu\in (0,h_{\text{max}}]$. Then \\
(1) $a_1\in (0,a]\mapsto \Upsilon_{\mu,T,a_1}$ is continuous;\\
(2) Let $0<a_1<a_2\leq a$, then
$\frac{a_1^2}{a_2^2}\Upsilon_{\mu,T,a_2}<\Upsilon_{\mu,T,a_1}<0$.
\end{lemma}

The next  compactness lemma  is  useful in the study of the
autonomous problem as well as in the non-autonomous problem.

\begin{lemma}\label{lem3.6}
Let $N, a,\eta,p,q$ be as in (\ref{e1.1}), (\ref{e2.1}) hold,
$\mu\in (0,h_{\text{max}}]$, $a_1\in(0,a]$. $\{u_n\}\subset S(a_1)$
be a minimizing sequence with respect to $\Upsilon_{\mu,T,a_1}$.
Then, for
some subsequence, either\\
(i) $\{u_n\}$ is strongly convergent,\\
  or\\
(ii) There exists $\{y_n\}\subset \mathbb{R}^N$ with
$|y_n|\to\infty$ such that the sequence $v_n(x) = u_n(x + y_n)$ is
strongly convergent to a function $v\in S(a_1)$ with
$J_{\mu,T}(v)=\Upsilon_{\mu,T,a_1}$.
\end{lemma}

\begin{proof}
Noting that $\|\nabla u_n\|_2<R_0$ for $n$ large enough, there
exists $u\in H^1(\mathbb{R}^N)$ such that $u_n\rightharpoonup u$ in
$H^1(\mathbb{R}^N)$ for some subsequence. Now we consider the
following three possibilities.

(1) If $u\not\equiv0$ and $\|u\|_2=b\neq a_1$, we must have $b\in
(0,a_1)$. Setting $v_n=u_n-u$, $d_n=\|v_n\|_2$, and by using
\begin{equation*}
\|u_n\|_2^2=\|v_n\|_2^2+\|u\|_2^2+o_n(1),
\end{equation*}
we obtain that $\|v_n\|_2\to d$, where $a_1^2=d^2+b^2$. Noting that
$d_n\in (0,a_1)$ for $n$ large enough, and using the
Br\'{e}zis-Lieb Lemma (see \cite{Willem 1996}), Lemma \ref{lem3.5},
$\|\nabla u_n\|_2^2=\|\nabla u\|_2^2+\|\nabla v_n\|_2^2+o_n(1)$,
$\|\nabla u\|_2^2\leq \liminf_{n\to+\infty}\|\nabla u_n\|_2^2$,
$\tau$ is continuous and non-increasing, we obtain that
\begin{equation*}
\begin{split}
\Upsilon_{\mu,T,a_1}+o_n(1)=J_{\mu,T}(u_n)&=\frac{1}{2}\|\nabla
v_n\|_2^2-\frac{\mu}{q}\|v_n\|_q^q-\frac{\eta}{p}\tau(\|\nabla
u_n\|_2)\|v_n\|_p^p\\
&\quad + \frac{1}{2}\|\nabla
u\|_2^2-\frac{\mu}{q}\|u\|_q^q-\frac{\eta}{p}\tau(\|\nabla
u_n\|_2)\|u\|_p^p+o_n(1)\\
&\geq
J_{\mu,T}(v_n)+J_{\mu,T}(u)+o_n(1)\\
&\geq \Upsilon_{\mu,T,d_n}+\Upsilon_{\mu,T,b}+o_n(1)\\
&\geq
\frac{d_n^2}{a_1^2}\Upsilon_{\mu,T,a_1}+\Upsilon_{\mu,T,b}+o_n(1).
\end{split}
\end{equation*}
Letting $n\to+\infty$, we find that
\begin{equation*}
\begin{split}
\Upsilon_{\mu,T,a_1}&\geq
\frac{d^2}{a_1^2}\Upsilon_{\mu,T,a_1}+\Upsilon_{\mu,T,b}\\
&>
\frac{d^2}{a_1^2}\Upsilon_{\mu,T,a_1}+\frac{b^2}{a_1^2}\Upsilon_{\mu,T,a_1}
=\Upsilon_{\mu,T,a_1},
\end{split}
\end{equation*}
which is a contradiction. So this possibility can not exist.

(2) If $\|u\|_2=a_1$, then  $u_n\to u$  in $L^2(\mathbb{R}^N)$ and
thus $u_n\to u$  in $L^t(\mathbb{R}^N)$ for all $t\in (2,2^*)$.

Case $p<2^*$, then
\begin{equation*}
\begin{split}
\Upsilon_{\mu,T,a_1}&=\lim_{n\to+\infty}J_{\mu,T}(u_n)\\
&= \lim_{n\to+\infty}\left(\frac{1}{2}\|\nabla
u_n\|_2^2-\frac{\mu}{q}\|u_n\|_q^q-\frac{\eta}{p}\tau(\|\nabla
u_n\|_2)\|u_n\|_p^p\right)\\
&\geq J_{\mu,T}(u).
\end{split}
\end{equation*}
As $u\in S(a_1)$, we infer that $J_{\mu,T}(u)=\Upsilon_{\mu,T,a_1}$,
then $\|\nabla u_n\|_2\to \|\nabla u\|_2$ and thus $u_n\to u$ in
$H^1(\mathbb{R}^N)$, which implies that (i) occurs.

Case $p=2^*$, noting that $\|\nabla v_n\|_2\leq \|\nabla
u_n\|_2<R_0$ for $n$ large enough, and using the Soblev inequality,
we have
\begin{equation}\label{e3.3}
\begin{split}
J_{\mu,T}(v_n)=J_{\mu}(v_n)&=\frac{1}{2}\int_{\mathbb{R}^N}|\nabla
v_n|^2dx-\frac{\mu}{q}\int_{\mathbb{R}^N}|v_n|^{q}dx
-\frac{\eta}{p}\int_{\mathbb{R}^N}|v_n|^{p}dx\\
&\geq \frac{1}{2}\|\nabla
v_n\|_2^2-\frac{\eta}{2^*}\frac{1}{S^{\frac{2^*}{2}}}\|\nabla
v_n\|_2^{2^*}+o_n(1)\\
&=\|\nabla
v_n\|_2^2\left(\frac{1}{2}-\frac{\eta}{2^*}\frac{1}{S^{\frac{2^*}{2}}}\|\nabla
v_n\|_2^{2^*-2}\right)+o_n(1)\\
&\geq \|\nabla
v_n\|_2^2\left(\frac{1}{2}-\frac{\eta}{2^*}\frac{1}{S^{\frac{2^*}{2}}}R_0^{2^*-2}\right)
+o_n(1)\\
&= \|\nabla
v_n\|_2^2\frac{1}{q}h_{\text{max}}C_{N,q}^qa^{q(1-\gamma_q)}R_0^{q\gamma_q-2}
+o_n(1),
\end{split}
\end{equation}
because
$w_a(R_0)=\frac{1}{2}-\frac{\eta}{2^*}\frac{1}{S^{\frac{2^*}{2}}}R_0^{2^*-2}
-\frac{1}{q}h_{\text{max}}C_{N,q}^qa^{q(1-\gamma_q)}R_0^{q\gamma_q-2}=0$.
Now we remember that
\begin{equation}\label{e3.4}
\Upsilon_{\mu,T,a_1}\leftarrow J_{\mu,T}(u_n)\geq
J_{\mu,T}(v_n)+J_{\mu,T}(u)+o_n(1).
\end{equation}
Since $u\in S(a_1)$, we have  $J_{\mu,T}(u)\geq
\Upsilon_{\mu,T,a_1}$, which combined with (\ref{e3.3}) and
(\ref{e3.4}) gives that $\|\nabla v_n\|_2^2\to 0$ and then $u_n\to
u$ in $H^1(\mathbb{R}^N)$. This implies that (i) occurs.

(3) If $u\equiv 0$, that is, $u_n\rightharpoonup 0$ in
$H^1(\mathbb{R}^N)$. We claim that there exist $R, \beta>0$ and
$\{y_n\}\subset \mathbb{R}^N$ such that
\begin{equation}\label{e3.5}
\int_{B_R(y_n)}|u_n|^2dx\geq \beta,\ \text{for\ all\ }n.
\end{equation}
Indeed, otherwise we must have $u_n\to 0$ in $L^t(\mathbb{R}^N)$ for
all $t\in(2, 2^*)$. Thus, for $p<2^*$,  $J_{\mu,T}(u_n)\geq
\frac{1}{2}\|\nabla u_n\|_2^2+o_n(1)$, which contradicts
$J_{\mu,T}(u_n)\to \Upsilon_{\mu,T,a_1}<0$. For $p=2^*$, similarly
to (\ref{e3.3}), we obtain that
\begin{equation*}
J_{\mu,T}(u_n)\geq
\|\nabla
u_n\|_2^2\frac{1}{q}h_{\text{max}}C_{N,q}^qa^{q(1-\gamma_q)}R_0^{q\gamma_q-2}
+o_n(1).
\end{equation*}
We also get a contradiction in this case. Hence, in all cases,
(\ref{e3.5}) holds and $|y_n|\to +\infty$ obviously. From this,
considering $\bar{u}_n(x) = u_n(x+y_n)$, clearly
$\{\bar{u}_n\}\subset S(a_1)$ and it is also a minimizing sequence
with respect to $\Upsilon_{\mu,T,a_1}$. Moreover, there exists
$\bar{u}\in H^1(\mathbb{R}^N)\backslash\{0\}$ such that
$\bar{u}_n\rightharpoonup \bar{u}$ in $H^1(\mathbb{R}^N)$. Following
as in the first two possibilities of the proof, we derive that
$\bar{u}_n\to\bar{u}$ in $H^1(\mathbb{R}^N)$, which implies that
(ii) occurs. This proves the lemma.
\end{proof}

\begin{lemma}\label{lem3.7}
Let $N, a,\eta,p,q$ be as in (\ref{e1.1}), $\mu\in
(0,h_{\text{max}}]$, $a_1\in(0,a]$, (\ref{e2.1}) hold. Then
$\Upsilon_{\mu,T,a_1}$ is attained.
\end{lemma}

\begin{proof}
By Lemmas \ref{lem3.2}-\ref{lem3.1}, there exists a bounded
minimizing sequence $\{u_n\} \subset S(a_1)$ satisfying
$J_{\mu,T}(u_n)\to \Upsilon_{\mu,T,a_1}$ as $n\to+\infty$. Now,
applying Lemma \ref{lem3.6}, there exists $u\in S(a_1)$ such that
$J_{\mu,T}(u)= \Upsilon_{\mu,T,a_1}$. The proof is complete.
\end{proof}

An immediate consequence of Lemma \ref{lem3.7} is the following
corollary.

\begin{corollary}\label{cor3.1}
Let $N, a,\eta,p,q$ be as in (\ref{e1.1}), (\ref{e2.1}) hold. Fix
$a_1\in (0,a]$ and let $0 < \mu_1 < \mu_2\leq h_{\text{max}}$. Then
$\Upsilon_{\mu_2,T,a_1} < \Upsilon_{\mu_1,T,a_1}$.
\end{corollary}

\begin{proof}
Let $u\in S(a_1)$ satisfying
$J_{\mu_1,T}(u)=\Upsilon_{\mu_1,T,a_1}$. Then,
$\Upsilon_{\mu_2,T,a_1}\leq J_{\mu_2,T}(u)<J_{\mu_1,T}(u) =
\Upsilon_{\mu_1,T,a_1}$.
\end{proof}

\section{Proof of Theorem \ref{thm1.1}}
\setcounter{section}{4} \setcounter{equation}{0}

In this section, we  first prove some properties of the functional
$E_{\epsilon,T}$ defined in  (\ref{e2.2}) restricted to the sphere
$S(a)$,  and then give the proof of Theorem \ref{thm1.1}.

Denote
\begin{equation*}
J_{\text{max},T}:=J_{h_{\text{max}},T}, \
\Upsilon_{\text{max},T,a}:=\Upsilon_{h_{\text{max}},T,a},
\end{equation*}
and
\begin{equation*}
J_{\infty,T}:=J_{h_{\infty},T}, \
\Upsilon_{\infty,T,a}:=\Upsilon_{h_{\infty},T,a}.
\end{equation*}
It is obvious that $J_{\infty,T}(u)\geq J_{\text{max},T}(u)$ and
$E_{\epsilon,T}(u)\geq J_{\text{max},T}(u)$ for any $u\in S(a)$. By
Lemma \ref{lem3.2}, the definition
\begin{equation*}
\Gamma_{\epsilon,T,a}:=\inf_{u\in S(a)}E_{\epsilon,T}(u)
\end{equation*}
is well defined and  $\Gamma_{\epsilon,T,a}\geq
\Upsilon_{\text{max},T,a}$.

The next lemma  establishes some crucial relations involving the
levels $\Gamma_{\epsilon,T,a}$, $\Upsilon_{\infty,T,a}$ and
$\Upsilon_{\text{max},T,a}$.

\begin{lemma}\label{lem4.2}
Let $N, a,\eta,p,q, h, \epsilon$ be as in (\ref{e1.1}), (\ref{e2.1})
hold. Then
$$\limsup_{\epsilon\to 0^+}\Gamma_{\epsilon,T,a}\leq
\Upsilon_{\text{max},T,a}<\Upsilon_{\infty,T,a}<0.$$
\end{lemma}

\begin{proof}
By Lemma \ref{lem3.7}, choose  $u \in  S(a)$ such that
$J_{\text{max},T}(u)=\Upsilon_{\text{max},T,a}$. Then,
\begin{equation*}
\begin{split}
\Gamma_{\epsilon,T,a}\leq E_{\epsilon,
T}(u)&=\frac{1}{2}\int_{\mathbb{R}^N}|\nabla
u|^2dx-\frac{1}{q}\int_{\mathbb{R}^N}h(\epsilon
x)|u|^{q}dx\\
&\qquad-\frac{\eta}{p}\tau(\|\nabla
u\|_2)\int_{\mathbb{R}^N}|u|^{p}dx.
\end{split}
\end{equation*}
Letting $\epsilon\to 0^+$, by the Lebesgue dominated convergence
theorem, we deduce that
\begin{equation*}
\limsup_{\epsilon\to 0^+}\Gamma_{\epsilon,T,a}\leq
\limsup_{\epsilon\to 0^+} E_{\epsilon, T}(u)=J_{h(0),T}(u)=
J_{\text{max},T}(u)=\Upsilon_{\text{max},T,a},
\end{equation*}
which combined with Lemma \ref{lem3.4} and Corollary \ref{cor3.1}
completes the proof.
\end{proof}

By Lemma \ref{lem4.2}, there exists $\epsilon_1>0$ such that
$\Gamma_{\epsilon,T,a}<\Upsilon_{\infty,T,a}$ for all $\epsilon\in
(0,\epsilon_1)$. In the following, we always assume that
$\epsilon\in (0,\epsilon_1)$. Similarly to the proof of  Lemma
\ref{lem3.1}, we have the following result, whose proof is omitted.

\begin{lemma}\label{lem4.4}
Let $N, a,\eta,p,q, h, \epsilon$ be as in (\ref{e1.1}),
$\epsilon\in(0,\epsilon_1)$, (\ref{e2.1})
hold. Then\\
(1) $E_{\epsilon,T}\in C^1(H^1(\mathbb{R}^N),\mathbb{R})$;\\
(2) If $u\in S(a)$ such that $E_{\epsilon,T}(u)<0$, then $\|\nabla
u\|_2<R_0$ and $E_{\epsilon,T}(v) = E_{\epsilon} (v)$ for all $v$
satisfying $\|v\|_2\leq a$ and being in a small neighborhood of $u$
in $H^1(\mathbb{R}^N)$.
\end{lemma}

The next two lemmas will be used to prove the $(PS)$ condition for
$E_{\epsilon,T}$ restricts to $S(a)$ at some levels.

\begin{lemma}\label{lem4.1}
Let $N, a,\eta,p,q, h$ be as in (\ref{e1.1}), $\epsilon\in
(0,\epsilon_1)$, (\ref{e2.1}) hold. Assume  $\{u_n\}\subset S(a)$
such that $E_{\epsilon,T}(u_n)\to c$ as $n\to+\infty$ with
$c<\Upsilon_{\infty,T,a}$. If $u_n\rightharpoonup u$ in
$H^1(\mathbb{R}^N)$, then $u\not\equiv 0$.
\end{lemma}

\begin{proof}
Assume by contradiction that $u\equiv0$. Then,
\begin{equation*}
c+o_n(1)=E_{\epsilon,T}(u_n)=J_{\infty,T}(u_n)+\frac{1}{q}\int_{\mathbb{R}^N}(h_{\infty}-h(\epsilon
x))|u_n|^qdx.
\end{equation*}
By ($h_2$), for any given $\delta>0$, there exists $R>0$ such that
$h_{\infty}\geq h(x)-\delta$ for all $|x|\geq R$. Hence,
\begin{equation*}
\begin{split}
c+o_n(1)=E_{\epsilon,T}(u_n)&\geq
J_{\infty,T}(u_n)+\frac{1}{q}\int_{B_{R/\epsilon}(0)}(h_{\infty}
-h(\epsilon
x))|u_n|^qdx\\
&\qquad-\frac{\delta}{q}\int_{B^c_{R/\epsilon}(0)}|u_n|^qdx.
\end{split}
\end{equation*}
Recalling that $\{u_n\}$ is bounded in $H^1(\mathbb{R}^N)$ and
$u_n\to  0$ in $L^t(B_{R/\epsilon}(0))$ for all $t\in[1, 2^*)$, it
follows that
\begin{equation*}
c+o_n(1)=E_{\epsilon,T}(u_n)\geq J_{\infty,T}(u_n)-\delta C+o_n(1)
\end{equation*}
for some $C > 0$. Since $\delta > 0$ is arbitrary, we deduce that
$c\geq \Upsilon_{\infty,T,a}$, which is a contradiction. Thus,
$u\not\equiv 0$.
\end{proof}

\begin{lemma}\label{lem4.5}
Let $N, a,\eta,p,q, h$ be as in (\ref{e1.1}), $\epsilon\in
(0,\epsilon_1)$, (\ref{e2.1}) hold. If $p=2^*$, we further assume
(\ref{e2.4}) hold. Let $\{u_{n}\}$ be a $(PS)_c$ sequence of
$E_{\epsilon,T}$ restricted to $S(a)$ with $c <
\Upsilon_{\infty,T,a}$ and let $u_n\rightharpoonup u_\epsilon$ in
$H^1(\mathbb{R}^N)$. If $u_n\not\rightarrow u_{\epsilon}$ in
$H^1(\mathbb{R}^N)$, there exists $\beta>0$ independent of
$\epsilon\in (0,\epsilon_1)$ such that
\begin{equation*}
\limsup_{n\to +\infty}\|u_n-u_{\epsilon}\|_2\geq \beta.
\end{equation*}
\end{lemma}

\begin{proof}
By Lemma \ref{lem4.4}, we must have $\|\nabla u_n\|_2<R_0$ for $n$
large enough, and so, $\{u_n\}$ is also a $(PS)_c$ sequence of
$E_{\epsilon}$ restricted to $S(a)$. Hence,
\begin{equation*}
E_{\epsilon}(u_n)\to c\ \text{and}\
\|E_{\epsilon}|'_{S(a)}(u_n)\|\to 0\ \text{as}\ n\to+\infty.
\end{equation*}
Setting the functional $\Psi: H^1(\mathbb{R}^N)\to \mathbb{R}$ given
by
\begin{equation*}
\Psi(u)=\frac{1}{2}\int_{\mathbb{R}^N}|u|^2dx,
\end{equation*}
it follows that $S(a) =\Psi^{-1}(a^2/2)$. Then, by Willem
(\cite{Willem 1996}, Proposition 5.12), there exists
$\{\lambda_n\}\subset \mathbb{R}$ such that
\begin{equation}\label{e4.8}
\|E'_{\epsilon}(u_n)-\lambda_n\Psi'(u_n)\|_{H^{-1}(\mathbb{R}^N)}\to
0\ \text{as}\ n\to+\infty.
\end{equation}

By the boundedness of $\{u_n\}$ in $H^1(\mathbb{R}^N)$, we know
$\{\lambda_n\}$ is bounded and thus, for some subsequence, there
exists $\lambda_{\epsilon}$ such that $\lambda_n\to
\lambda_{\epsilon}$ as $n\to+\infty$. This together with
(\ref{e4.8}) leads to
\begin{equation}\label{e4.10}
E'_{\epsilon}(u_\epsilon)-\lambda_\epsilon\Psi'(u_\epsilon)=0\text{\
in \ }H^{-1}(\mathbb{R}^N)
\end{equation}
and then
\begin{equation}\label{e4.9}
\|E'_{\epsilon}(v_n)-\lambda_\epsilon\Psi'(v_n)\|_{H^{-1}(\mathbb{R}^N)}\to
0\ \text{as}\ n\to+\infty,
\end{equation}
where $v_n:=u_n-u_\epsilon$. By direct calculations, we get that
\begin{equation*}
\begin{split}
\Upsilon_{\infty,T,a}&>\lim_{n\to+\infty}E_{\epsilon}(u_n)\\
&=\lim_{n\to+\infty}\left(E_{\epsilon}(u_n)-\frac{1}{2}E_{\epsilon}'(u_n)u_n+\frac{1}{2}\lambda_n\|u_n\|_2^2+o_n(1)\right)\\
&=\lim_{n\to+\infty}\left[\left(\frac{1}{2}-\frac{1}{q}\right)\int_{\mathbb{R}^N}h(\epsilon
x)|u_n|^qdx\right.\\
&\qquad\qquad\qquad\left.+\left(\frac{1}{2}-\frac{1}{p}\right)\eta\int_{\mathbb{R}^N}|u_n|^pdx+\frac{1}{2}\lambda_na^2+o_n(1)\right]\\
&\geq \frac{1}{2}\lambda_{\epsilon}a^2,
\end{split}
\end{equation*}
which implies that
\begin{equation}\label{e4.11}
\lambda_{\epsilon}\leq \frac{2\Upsilon_{\infty,T,a}}{a^2}<0\
\text{for\ all\ }\epsilon\in (0,\epsilon_1).
\end{equation}
By (\ref{e4.9}), we know
\begin{equation}\label{e4.15}
\int_{\mathbb{R}^N}|\nabla
v_n|^2dx-\lambda_{\epsilon}\int_{\mathbb{R}^N}|v_n|^2dx-\int_{\mathbb{R}^N}h(\epsilon
x)|v_n|^{q}dx-\eta\int_{\mathbb{R}^N}|v_n|^{p}dx=o_n(1),
\end{equation}
which combined with (\ref{e4.11}) gives that
\begin{equation}\label{e4.3}
\begin{split}
\int_{\mathbb{R}^N}|\nabla
v_n|^2dx&-\frac{2\Upsilon_{\infty,T,a}}{a^2}\int_{\mathbb{R}^N}|v_n|^2dx\\
&\leq
h_{\text{max}}\int_{\mathbb{R}^N}|v_n|^{q}dx+\eta\int_{\mathbb{R}^N}|v_n|^{p}dx+o_n(1).
\end{split}
\end{equation}

If $u_n\not\rightarrow u_{\epsilon}$ in $H^1(\mathbb{R}^N)$, that
is, $v_n\not\rightarrow 0$ in $H^1(\mathbb{R}^N)$, by (\ref{e4.3})
and the Sobolev inequality, we deduce that
\begin{equation*}
\begin{split}
\int_{\mathbb{R}^N}|\nabla
v_n|^2dx&-\frac{2\Upsilon_{\infty,T,a}}{a^2}\int_{\mathbb{R}^N}|v_n|^2dx\\
&\leq h_{\text{max}}C_{N,q}^q\|v_n\|^q+\eta
C_{N,p}^p\|v_n\|^p+o_n(1).
\end{split}
\end{equation*}
So there exists $C>0$ independent of $\epsilon$ such that
$\|v_n\|\geq C$ and then by (\ref{e4.3})
\begin{equation}\label{e4.4}
\limsup_{n\to
+\infty}\left(h_{\text{max}}\int_{\mathbb{R}^N}|v_n|^{q}dx+\eta\int_{\mathbb{R}^N}|v_n|^{p}dx\right)\geq
C.
\end{equation}
Case $p<2^*$, by (\ref{e4.4}) and  the Gagliardo-Nirenberg
inequality (\ref{e2.9}), there exists $\beta
> 0$ independent of $\epsilon\in (0,\epsilon_1)$ such that
\begin{equation}\label{e4.5}
\limsup_{n\to +\infty}\|v_n\|_2\geq \beta.
\end{equation}
Case $p=2^*$, if
\begin{equation*}
\limsup_{n\to +\infty}\int_{\mathbb{R}^N}|v_n|^{q}dx\geq C
\end{equation*}
for some $C>0$ independent of $\epsilon$, we  obtain (\ref{e4.5}) as
well. If
\begin{equation*} \liminf_{\epsilon\to 0^+}\limsup_{n\to
+\infty}\int_{\mathbb{R}^N}|v_n|^{q}dx=0\ \text{and}\
\liminf_{\epsilon\to 0^+}\limsup_{n\to +\infty}\|v_n\|_2=0,
\end{equation*}
by (\ref{e4.4}) we have
\begin{equation*}
\liminf_{\epsilon\to 0^+}\limsup_{n\to
+\infty}\int_{\mathbb{R}^N}|v_n|^{p}dx\geq C,
\end{equation*}
and by (\ref{e4.15}) we have
\begin{equation*}
\liminf_{\epsilon\to 0^+}\limsup_{n\to
+\infty}\int_{\mathbb{R}^N}|\nabla v_n|^2dx=\liminf_{\epsilon\to
0^+}\limsup_{n\to +\infty}\eta\int_{\mathbb{R}^N}|v_n|^{p}dx.
\end{equation*}
Applying the Sobolev inequality to the above equality, we obtain
that
\begin{equation*}
\liminf_{\epsilon\to 0^+}\limsup_{n\to +\infty}\|\nabla
v_n\|_2^2=\liminf_{\epsilon\to 0^+}\limsup_{n\to
+\infty}\eta\|v_n\|_{2^*}^{2^*}\leq\liminf_{\epsilon\to
0^+}\limsup_{n\to +\infty}\eta S^{-2^*/2}\|\nabla v_n\|_2^{2^*},
\end{equation*}
which implies that
\begin{equation*}
R_0\geq \liminf_{\epsilon\to 0^+}\limsup_{n\to +\infty}\|\nabla
v_n\|_2\geq \eta^{-\frac{N-2}{4}}S^{N/4}.
\end{equation*}
That contradicts  $R_0<\eta^{-\frac{N-2}{4}}S^{N/4}$, that is,
assumption (\ref{e2.4}). So we must have (\ref{e4.5}) for the case
$p=2^*$. The proof is complete.
\end{proof}

Now we give the compactness lemma.

\begin{lemma}\label{lem4.3}
Let $N, a,\eta,p,q, h$ be as in (\ref{e1.1}), $\epsilon\in
(0,\epsilon_1)$, $\beta$ be as in Lemma \ref{lem4.5}, $$0<\rho_0\leq
\min\left\{\Upsilon_{\infty,T,a}
-\Upsilon_{\text{max},T,a},\frac{\beta^2}{a^2}(\Upsilon_{\infty,T,a}-\Upsilon_{\text{max},T,a})\right\},$$
and (\ref{e2.1}) hold. If $p=2^*$, we further assume (\ref{e2.4})
hold. Then $E_{\epsilon,T}$ satisfies the $(PS)_c$ condition
restricted to $S(a)$ if $c <\Upsilon_{\text{max},T,a}+\rho_0$.
\end{lemma}

\begin{proof}
Let $\{u_{n}\}\subset S(a)$ be a $(PS)_c$ sequence of
$E_{\epsilon,T}$ restricted to $S(a)$. Noting that
$c<\Upsilon_{\infty,T,a}<0$, by Lemma \ref{lem4.4}, $\{u_n\}$ is
bounded in $H^1(\mathbb{R}^N)$. Let $u_n\rightharpoonup u_\epsilon$
in $H^1(\mathbb{R}^N)$. By Lemma \ref{lem4.1}, $u_\epsilon\not\equiv
0$. Set $v_n:=u_n-u_{\epsilon}$. If $u_n\to u_\epsilon$ in
$H^1(\mathbb{R}^N)$, the proof is complete. If $u_n\not\rightarrow
u_{\epsilon}$ in $H^1(\mathbb{R}^N)$ for some $\epsilon\in
(0,\epsilon_1)$, by Lemma \ref{lem4.5},
\begin{equation*}
\limsup_{n\to +\infty}\|v_n\|_2\geq \beta.
\end{equation*}
Set $b=\|u_{\epsilon}\|_2,\ d_n = \|v_n\|_2$ and suppose that
$\|v_n\|_2\to d$, then  we get $d\geq \beta>0$ and  $a^2 = b^2 +
d^2$. From $d_n\in(0, a)$ for $n$ large enough, we have
\begin{equation}\label{e4.6}
c+o_n(1)=E_{\epsilon,T}(u_n)\geq
E_{\epsilon,T}(v_n)+E_{\epsilon,T}(u_{\epsilon})+o_n(1).
\end{equation}
Since $v_n\rightharpoonup 0$ in $H^1(\mathbb{R}^N)$, similarly to
the proof of Lemma \ref{lem4.1}, we deduce that
\begin{equation}\label{e4.7}
E_{\epsilon,T}(v_n)\geq J_{\infty,T}(v_n)-\delta C+o_n(1)
\end{equation}
for any $\delta>0$, where $C>0$ is a constant independent of
$\delta,\ \epsilon$ and $n$. By (\ref{e4.6}) and (\ref{e4.7}), we
obtain that
\begin{equation*}
\begin{split}
c+o_n(1)=E_{\epsilon,T}(u_n)&\geq
J_{\infty,T}(v_n)+E_{\epsilon,T}(u_{\epsilon})-\delta C+o_n(1)\\
&\geq \Upsilon_{\infty,T,d_n}+\Upsilon_{\text{max},T,b}-\delta
C+o_n(1).
\end{split}
\end{equation*}
Letting $n\to +\infty$, by Lemma \ref{lem3.5} and the arbitrariness
of $\delta>0$, we obtain that
\begin{equation*}
\begin{split}
c\geq \Upsilon_{\infty,T,d}+\Upsilon_{\text{max},T,b}&\geq
\frac{d^2}{a^2}\Upsilon_{\infty,T,a}+\frac{b^2}{a^2}\Upsilon_{\text{max},T,a}\\
&=\Upsilon_{\text{max},T,a}+\frac{d^2}{a^2}(\Upsilon_{\infty,T,a}-\Upsilon_{\text{max},T,a})\\
&\geq
\Upsilon_{\text{max},T,a}+\frac{\beta^2}{a^2}(\Upsilon_{\infty,T,a}-\Upsilon_{\text{max},T,a}),
\end{split}
\end{equation*}
which contradicts
$c<\Upsilon_{\text{max},T,a}+\frac{\beta^2}{a^2}(\Upsilon_{\infty,T,a}-\Upsilon_{\text{max},T,a})$.
Thus, we must have $u_n \to u_{\epsilon}$ in $H^1(\mathbb{R}^N)$.
\end{proof}

In what follows, let us fix $\tilde{\rho}, \tilde{r} > 0$
satisfying:

\smallskip

$\bullet$\ $\overline{B_{\tilde{\rho}}(a_i)}\cap
\overline{B_{\tilde{\rho}}(a_j)}=\emptyset$ for $i\neq j$ and
$i,j\in\{1,\cdots,l\}$;

\smallskip

$\bullet$\ $\cup_{i=1}^{l}B_{\tilde{\rho}}(a_i) \subset
B_{\tilde{r}}(0)$;

\smallskip

$\bullet$\
$K_{\frac{\tilde{\rho}}{2}}=\cup_{i=1}^{l}\overline{B_{\frac{\tilde{\rho}}{2}}(a_i)}$.

We also set the function $Q_\epsilon:
H^1(\mathbb{R}^N)\backslash\{0\}\to \mathbb{R}^N$ by
\begin{equation*}
Q_\epsilon(u):=\frac{\int_{\mathbb{R}^N}\chi(\epsilon
x)|u|^2dx}{\int_{\mathbb{R}^N}|u|^2dx},
\end{equation*}
where $\chi:\mathbb{R}^N\to \mathbb{R}^N$ is given by
\begin{equation*}
\chi(x):=\left\{
\begin{array}{ll}
x, &\ \text{if}\ |x|\leq \tilde{r},\\
\tilde{r}\frac{x}{|x|}, &\ \text{if}\ |x|>\tilde{r}.
\end{array}
\right.
\end{equation*}

The next two lemmas will be useful to get important ($PS$) sequences
for $E_{\epsilon,T}$ restricted to $S(a)$.

\begin{lemma}\label{lem4.6}
Let $N, a,\eta,p,q, h$ be as in (\ref{e1.1}), and (\ref{e2.1}) hold.
Then there exist $\epsilon_2\in(0,\epsilon_1], \rho_1\in(0,\rho_0]$
such that if $\epsilon\in(0, \epsilon_2)$, $u\in S(a)$ and
$E_{\epsilon,T}(u)\leq \Upsilon_{\text{max},T,a}+\rho_1$, then
\begin{equation*}
Q_\epsilon(u)\in  K_{\frac{\tilde{\rho}}{2}}.
\end{equation*}
\end{lemma}

\begin{proof}
If the lemma does not occur, there must be $\rho_n\to 0,
\epsilon_n\to 0$ and $\{u_n\}\subset S(a)$ such that
\begin{equation}\label{e4.16}
E_{\epsilon_n,T}(u_n)\leq \Upsilon_{\text{max},T,a}+\rho_n\
\text{and}\ Q_{\epsilon_n}(u_n)\not\in K_{\frac{\tilde{\rho}}{2}}.
\end{equation}
Consequently,
\begin{equation*}
\Upsilon_{\text{max},T,a}\leq  J_{\text{max},T}(u_n)\leq
E_{\epsilon_n,T}(u_n)\leq \Upsilon_{\text{max},T,a}+\rho_n,
\end{equation*}
then
\begin{equation*}
\{u_n\}\subset S(a)\ \text{and}\ J_{\text{max},T}(u_n)\to
\Upsilon_{\text{max},T,a}.
\end{equation*}
According to Lemma \ref{lem3.6}, we have two cases:\\
(i)\ $u_n\to u$ in $H^1(\mathbb{R}^N)$ for some $u\in S(a)$,\\
or\\
(ii)\ There exists $\{y_n\}\subset \mathbb{R}^N$ with $|y_n|\to
+\infty$ such that $v_n(x) = u_n(x+ y_n)$ converges in
$H^1(\mathbb{R}^N)$ to some $v\in S(a)$.

Analysis of (i): By the Lebesgue dominated convergence theorem,
\begin{equation*}
Q_{\epsilon_n}(u_n)=\frac{\int_{\mathbb{R}^N}\chi(\epsilon_n
x)|u_n|^2dx}{\int_{\mathbb{R}^N}|u_n|^2dx}\to
\frac{\int_{\mathbb{R}^N}\chi(0)|u|^2dx}{\int_{\mathbb{R}^N}|u|^2dx}=0\in
K_{\frac{\tilde{\rho}}{2}}.
\end{equation*}
From this, $Q_{\epsilon_n}(u_n)\in K_{\frac{\tilde{\rho}}{2}}$ for
$n$ large enough, which contradicts $Q_{\epsilon_n}(u_n)\not\in
K_{\frac{\tilde{\rho}}{2}}$ in (\ref{e4.16}).

Analysis of (ii): Now we will study two cases: (I)
$|\epsilon_ny_n|\to +\infty$ and (II) $\epsilon_ny_n\to y$ for some
$y\in \mathbb{R}^N$.

If (I) holds, the limit $v_n \to v$ in $H^1(\mathbb{R}^N)$ provides
\begin{equation*}
\begin{split}
E_{\epsilon_n,T}(u_n)&=\frac{1}{2}\int_{\mathbb{R}^N}|\nabla
v_n|^2dx-\frac{1}{q}\int_{\mathbb{R}^N}h(\epsilon_n
x+\epsilon_ny_n)|v_n|^{q}dx\\
&\qquad-\frac{\eta}{p}\tau(\|\nabla
v_n\|_2)\int_{\mathbb{R}^N}|v_n|^{p}dx\\
&\to J_{\infty,T}(v).
\end{split}
\end{equation*}
Since $E_{\epsilon_n,T}(u_n)\leq \Upsilon_{\text{max},T,a}+\rho_n$,
we deduce that
\begin{equation*}
\Upsilon_{\infty,T,a}\leq J_{\infty,T}(v)\leq
\Upsilon_{\text{max},T,a},
\end{equation*}
which contradicts $\Upsilon_{\infty,T,a}>\Upsilon_{\text{max},T,a}$
in Lemma \ref{lem4.2}.

Now if (II) holds, then
\begin{equation*}
\begin{split}
E_{\epsilon_n,T}(u_n)&=\frac{1}{2}\int_{\mathbb{R}^N}|\nabla
v_n|^2dx-\frac{1}{q}\int_{\mathbb{R}^N}h(\epsilon_n
x+\epsilon_ny_n)|v_n|^{q}dx\\
&\qquad-\frac{\eta}{p}\tau(\|\nabla
v_n\|_2)\int_{\mathbb{R}^N}|v_n|^{p}dx\\
&\to J_{h(y),T}(v),
\end{split}
\end{equation*}
which combined with $E_{\epsilon_n,T}(u_n)\leq
\Upsilon_{\text{max},T,a}+\rho_n$ gives  that
\begin{equation*}
\Upsilon_{h(y),T,a}\leq J_{h(y),T}(v)\leq \Upsilon_{\text{max},T,a}.
\end{equation*}
By Corollary \ref{cor3.1}, we must have $h(y)=h_{\text{max}}$ and
$y=a_i$ for some $i=1,2,\cdots, l$. Hence,
\begin{equation*}
\begin{split}
Q_{\epsilon_n}(u_n)=\frac{\int_{\mathbb{R}^N}\chi(\epsilon_n
x)|u_n|^2dx}{\int_{\mathbb{R}^N}|u_n|^2dx}
&=\frac{\int_{\mathbb{R}^N}\chi(\epsilon_n
x+\epsilon_ny_n)|v_n|^2dx}{\int_{\mathbb{R}^N}|v_n|^2dx}\\
&\to
\frac{\int_{\mathbb{R}^N}\chi(y)|v|^2dx}{\int_{\mathbb{R}^N}|v|^2dx}=a_i\in
K_{\frac{\tilde{\rho}}{2}},
\end{split}
\end{equation*}
which implies that $Q_{\epsilon_n}(u_n)\in
K_{\frac{\tilde{\rho}}{2}}$ for $n$ large enough. That contradicts
(\ref{e4.16}). The proof is complete.
\end{proof}

From now on, we will use the following notations:

\smallskip

$\bullet$\ $\theta_{\epsilon}^i:=\{u\in S(a):|Q_\epsilon(u)-a_i|\leq
\tilde{\rho}\}$;

\smallskip

$\bullet$\ $\partial\theta_{\epsilon}^i:=\{u\in
S(a):|Q_\epsilon(u)-a_i|=\tilde{\rho}\}$;

\smallskip

 $\bullet$\
$\beta_{\epsilon}^{i}:=\inf_{u\in\theta_{\epsilon}^i}E_{\epsilon,T}(u)$;

\smallskip

$\bullet$\
$\tilde{\beta}_{\epsilon}^{i}:=\inf_{u\in\partial\theta_{\epsilon}^i}E_{\epsilon,T}(u)$.

\begin{lemma}\label{lem4.7}
Let $N, a,\eta,p,q, h$ be as in (\ref{e1.1}),  (\ref{e2.1}) hold,
$\epsilon_2$ and $\rho_1$ be obtained in Lemma \ref{lem4.6}. Then
there exists $\epsilon_3\in (0,\epsilon_2]$ such that
\begin{equation*}
\beta_{\epsilon}^i<\Upsilon_{\text{max},T,a}+\frac{\rho_1}{2}\
\text{and}\ \beta_{\epsilon}^{i}<\tilde{\beta}_{\epsilon}^{i},\
\text{for\ any\ }\epsilon\in(0,\epsilon_3).
\end{equation*}
\end{lemma}

\begin{proof}
Let $u\in S(a)$ be such that
\begin{equation*}
J_{\text{max},T}(u)=\Upsilon_{\text{max},T,a}.
\end{equation*}
For $1\leq i\leq l$, we define
\begin{equation*}
\hat{u}_{\epsilon}^{i}(x):=u\left(x-\frac{a_i}{\epsilon}\right),\
x\in \mathbb{R}^N.
\end{equation*}
Then $\hat{u}_{\epsilon}^{i}\in S(a)$ for all $\epsilon>0$ and
$1\leq i\leq l$. Direct calculations give that
\begin{equation*}
E_{\epsilon,T}(\hat{u}_{\epsilon}^{i})=\frac{1}{2}\int_{\mathbb{R}^N}|\nabla
u|^2dx-\frac{1}{q}\int_{\mathbb{R}^N}h(\epsilon
x+a_i)|u|^{q}dx-\frac{\eta}{p}\tau(\|\nabla
u\|_2)\int_{\mathbb{R}^N}|u|^{p}dx,
\end{equation*}
and then
\begin{equation}\label{e4.1}
\lim_{\epsilon\to
0^+}E_{\epsilon,T}(\hat{u}_{\epsilon}^{i})=J_{h(a_i),T}(u)=J_{\text{max},T}(u)
=\Upsilon_{\text{max},T,a}.
\end{equation}
Note that
\begin{equation*}
Q_{\epsilon}(\hat{u}_{\epsilon}^i)=\frac{\int_{\mathbb{R}^N}\chi(\epsilon
x+a_i)|u|^2dx}{\int_{\mathbb{R}^N}|u|^2dx}\to a_i\text{\ as}\
\epsilon\to 0^+.
\end{equation*}
So $\hat{u}_{\epsilon}^i\in \theta_{\epsilon}^i$ for $\epsilon$
small enough, which combined with (\ref{e4.1}) implies that there
exists $\epsilon_3\in(0,\epsilon_2]$ such that
\begin{equation*}
\beta_{\epsilon}^i<\Upsilon_{\text{max},T,a}+\frac{\rho_1}{2}\text{\
for\ any}\ \epsilon\in (0,\epsilon_3).
\end{equation*}

For any $v\in \partial\theta_{\epsilon}^i$, that is, $v\in S(a)$ and
$|Q_\epsilon(v)-a_i|=\tilde{\rho}$, we obtain that
$|Q_\epsilon(v)\not\in K_{\frac{\tilde{\rho}}{2}}$. Thus, by Lemma
\ref{lem4.6},
\begin{equation*}
E_{\epsilon,T}(v)>\Upsilon_{\text{max},T,a}+\rho_1,\ \text{for\
all}\ v\in \partial\theta_{\epsilon}^i\ \text{and}\ \epsilon\in
(0,\epsilon_3),
\end{equation*}
which implies that
\begin{equation*}
\tilde{\beta}_{\epsilon}^{i}=\inf_{v\in\partial\theta_{\epsilon}^i}
E_{\epsilon,T}(v)\geq \Upsilon_{\text{max},T,a}+\rho_1,\ \text{for\
all}\ \epsilon\in(0,\epsilon_3).
\end{equation*}
Thus,
\begin{equation*}
\beta_{\epsilon}^{i}<\tilde{\beta}_{\epsilon}^{i},\ \text{for\ all\
} \epsilon\in (0,\epsilon_3).
\end{equation*}
\end{proof}

Now we are ready to prove Theorem \ref{thm1.1}.

\smallskip

\textbf{Proof of Theorem \ref{thm1.1}}. Set
$\epsilon_0:=\epsilon_3$, where $\epsilon_3$ is obtained in Lemma
\ref{lem4.7}. Let $\epsilon\in (0,\epsilon_0)$. By Lemma
\ref{lem4.7}, for each $i\in \{1,2,\cdots, l\}$, we can use the
Ekeland's variational principle to find a sequence
$\{u_n^{i})\subset \theta_{\epsilon}^i$ satisfying
\begin{equation*}
E_{\epsilon,T}(u_n^i)\to \beta_{\epsilon}^i\ \text{and}\
\|E_{\epsilon,T}|_{S(a)}'(u_n^i)\|\to 0\ \text{as}\ n\to +\infty,
\end{equation*}
that is, $\{u_n^i\}_n$ is a $(PS)_{\beta_{\epsilon}^i}$ sequence for
$E_{\epsilon,T}$ restricted to $S(a)$. Since $\beta_{\epsilon}^i<
\Upsilon_{\text{max},T,a}+\rho_0$, it follows from  Lemma
\ref{lem4.3} that there exists  $u^i$ such that $u_n^i\to u^i$ in
$H^1(\mathbb{R}^N)$. Thus
\begin{equation*}
u^i\in \theta_{\epsilon}^i,\ E_{\epsilon,T}(u^i)=\beta_{\epsilon}^i\
\text{and}\ E_{\epsilon,T}|_{S(a)}'(u^i)=0.
\end{equation*}

As
\begin{equation*}
Q_{\epsilon}(u^i)\in \overline{B_{\tilde{\rho}}(a_i)},\
Q_{\epsilon}(u^j)\in \overline{B_{\tilde{\rho}}(a_j)},
\end{equation*}
and
\begin{equation*}
\overline{B_{\tilde{\rho}}(a_i)}\cap
\overline{B_{\tilde{\rho}}(a_j)}=\emptyset\ \text{for}\ i\neq j,
\end{equation*}
we conclude that $u^i\not\equiv u^j$ for $i\neq j$ while $1\leq i, j
\leq  l$. Therefore, $E_{\epsilon, T}$ has at least $l$ nontrivial
critical points for all $\epsilon\in  (0, \epsilon_0)$.

As $E_{\epsilon, T}(u^i)<0$ for any $i=1,2,\cdots,l$, by Lemma
\ref{lem4.4}, $u^i\, (i=1,2,\cdots,l)$ are in fact the critical
points of $E_{\epsilon}$ on $S(a)$ with $E_{\epsilon}(u^i)<0$ and
then there exists $\lambda_i\in \mathbb{R}$ such that
\begin{equation*}
-\Delta u^i=\lambda_i u^i+h(\epsilon x)|u^i|^{q-2}u^i+\eta
|u^i|^{p-2}u^i,\quad x\in \mathbb{R}^N.
\end{equation*}

By using $E_{\epsilon}(u^i)=\beta_{\epsilon}^i<0$ and
$E_{\epsilon}'(u^i)u^i=\lambda_ia^2$, we obtain that
\begin{equation*}
\frac{1}{2}\lambda_ia^2=E_{\epsilon}(u^i)+\left(\frac{1}{q}-\frac{1}{2}\right)
\int_{\mathbb{R}^N}h(\epsilon
x)|u^i|^qdx+\left(\frac{1}{p}-\frac{1}{2}\right)
\int_{\mathbb{R}^N}|u^i|^pdx,
\end{equation*}
which implies that $\lambda_i<0$ for $i=1,2,\cdots,l$. The proof is
complete.

\section{Orbital stability}
\setcounter{section}{5} \setcounter{equation}{0}

In this section we investigate the orbital stability of the
solutions obtained in Theorem \ref{thm1.1}. We first give the
definition of orbital stability.

\begin{definition}\label{def5.1}
A set $\Omega\subset H^1(\mathbb{R}^N)$ is orbitally stable under
the flow associated with the problem
\begin{equation}\label{e4.17}
\begin{cases}
i\frac{\partial \psi}{\partial t}+\Delta \psi+h(\epsilon
x)|\psi|^{q-2}\psi+\eta |\psi|^{p-2}\psi=0,\ t>0,x\in\mathbb{R}^N,\\
\psi(0,x)=u_0(x)
\end{cases}
\end{equation}
if for any $\theta>0$ there exists $\gamma>0$ such that for any
$u_0\in H^1(\mathbb{R}^N)$ satisfying
\begin{equation*}
\text{dist}_{H^1(\mathbb{R}^N)}(u_0,\Omega)<\gamma,
\end{equation*}
the solution $\psi(t, \cdot)$ of problem (\ref{e4.17}) with $\psi(0,
x) =u_0$ satisfies
\begin{equation*}
\sup_{t\in
\mathbb{R}^+}\text{dist}_{H^1(\mathbb{R}^N)}(\psi(t,\cdot),\Omega)<\theta.
\end{equation*}
\end{definition}

For any $i=1,2,\cdots,l$, we define
\begin{equation*}
\begin{split}
\Omega_i:&=\{v\in \theta_{\epsilon}^i:\
E_{\epsilon,T}|_{S(a)}'(v)=0,\
E_{\epsilon,T}(v)=\beta_{\epsilon}^i\}\\
&=\{v\in \theta_{\epsilon}^i:\ E_{\epsilon}|_{S(a)}'(v)=0,\
E_{\epsilon}(v)=\beta_{\epsilon}^i, \ \|\nabla v\|_2\leq R_0\}.
\end{split}
\end{equation*}
Next we show the stability of  the sets $\Omega_i\, (i=1,\cdots,l)$
in two cases $p<2^*$ or $p=2^*$.

\begin{theorem}\label{thm5.3}
Let $N,a,\eta,q,h,\epsilon_0$ be as in Theorem \ref{thm1.1},
$p<2^*$, $\epsilon\in (0,\epsilon_0)$, (\ref{e2.1}) hold. Then
$\Omega_i \, (i=1,\cdots,l)$ is orbitally stable under the flow
associated with the problem (\ref{e4.17}).
\end{theorem}

\begin{proof}
Letting $r_1$ be such that $\bar{g}_a(r_1)=\beta_{\epsilon}^i$, and
considering (\ref{e2.8}), the definition of $\Omega_i$ and
$\beta_{\epsilon}^i<0$, we know that
\begin{equation*}
\|\nabla v\|_2\leq  r_1<R_0,\ \text{for\ any\ }v\in\Omega_i.
\end{equation*}
Let $a_1>a$ be such that
$\bar{g}_{a_1}(R_0)=\frac{\beta_{\epsilon}^i}{2}$. There exists
$\delta>0$ such that if
\begin{equation*}
u_0\in H^1(\mathbb{R}^N)\ \text{and}\
\text{dist}_{H^1(\mathbb{R}^N)}(u_0,\Omega_i)<\delta,
\end{equation*}
then
\begin{equation*}
\|u_0\|_2\leq a_1,\ \|\nabla u_0\|_2\leq r_1+\frac{R_0-r_1}{2},\
E_{\epsilon,T}(u_0)\leq \frac{2}{3}\beta_{\epsilon}^i.
\end{equation*}

Denoting by $\psi(t,\cdot)$ the solution to (\ref{e4.17}) with
initial data $u_0$ and denoting by $[0, T_{\text{max}})$ the maximal
existence interval for $\psi(t,\cdot)$, we have classically that
either $\psi(t,\cdot)$ is globally defined for positive times, or
$\|\nabla \psi(t,\cdot)\|_2=+\infty$ as $t\to T_{\text{max}}^-$, see
(\cite{Tao-Visan-Zhang07}, Section 3). Set $\tilde{a}=\|u_0\|_2$.
Note that $\|\psi(t,\cdot)\|_2=\|u_0\|_2$ for all $t\in
(0,T_{\text{max}})$ by the conservation of the mass. If there exists
$\tilde{t}\in (0,T_{\text{max}})$ such that $\|\nabla
\psi(\tilde{t},\cdot)\|_2=R_0$, then
\begin{equation*}
E_{\epsilon}(\psi(\tilde{t},\cdot))=E_{\epsilon,T}(\psi(\tilde{t},\cdot))
\geq \bar{g}_{\tilde{a}}(R_0)\geq
\bar{g}_{a_1}(R_0)=\frac{\beta_{\epsilon}^i}{2},
\end{equation*}
which contradicts the conservation of the energy
\begin{equation*}
E_{\epsilon}(\psi(t,\cdot))=E_{\epsilon}(u_0)\leq
\frac{2}{3}\beta_{\epsilon}^i,\ \text{for\ all\ }t\in
(0,T_{\text{max}}).
\end{equation*}
Thus,
\begin{equation}\label{e4.18}
\|\nabla \psi(t,\cdot)\|_2<R_0\ \text{for\ all\ } t\in
[0,T_{\text{max}}),
\end{equation}
which implies that $\psi(t,\cdot)$ is globally defined for positive
times.

Next we prove that $\Omega_i$ is orbitally stable. The validity of
Lemma \ref{lem4.3} for complex valued function can be proved exactly
as in Theorem 3.1 in \cite{Hajaiej-Stuart04}. Thus, the orbital
stability of $\Omega_i$ can be proved by modifying the classical
Cazenave-Lions argument \cite{Cazenave-Lions82} (see also
\cite{Li-Zhao20}). For the completeness,  we give the proof here.
Suppose by contradiction that there exist sequences
$\{u_{0,n}\}\subset H^1(\mathbb{R}^N)$ and $\{t_{n}\}\subset
\mathbb{R}^+$ and a constant $\theta_0>0$ such that for all $n\geq
1$,
\begin{equation}\label{e44.1}
\inf_{v\in \Omega_i}\|u_{0,n}-v\|<\frac{1}{n}
\end{equation}
and
\begin{equation}\label{e44.2}
\inf_{v\in \Omega_i}\|\psi_n(t_n,\cdot)-v\|\geq \theta_0,
\end{equation}
where $\psi_n(t,\cdot)$ is the solution to (\ref{e4.17}) with
initial data $u_{0,n}$. By (\ref{e4.18}), there exists $n_0$  such
that for $n>n_0$ it holds that  $\|\nabla \psi_n(t,\cdot)\|_2<R_0$
for all $t\geq 0$.

By (\ref{e44.1}), there exists $\{v_n\}\subset \Omega_i$ such that
\begin{equation}\label{e44.3}
\|u_{0,n}-v_n\|<\frac{2}{n}.
\end{equation}
That $\{v_n\} \subset \Omega_i$ implies that $\{v_n\}\subset
\theta_{\epsilon}^i$ is a $(PS)_{\beta_{\epsilon}^i}$ sequence of
$E_{\epsilon,T}$ restricted to $S(a)$. From the proof of Theorem
\ref{thm1.1},  there exists $v\in \Omega_i$ such that
\begin{equation*}
\lim_{n\to +\infty}\|v_n-v\|=0,
\end{equation*}
which combined with (\ref{e44.3}) gives that
\begin{equation}\label{e4.19}
\lim_{n\to +\infty}\|u_{0,n}-v\|=0.
\end{equation}
Hence,
\begin{equation*}
\lim_{n\to +\infty}\|u_{0,n}\|_2=\|v\|_2=a,\ \ \lim_{n\to
+\infty}E_{\epsilon}(u_{0,n})=E_{\epsilon}(v)={\beta}_{\epsilon}^i<\tilde{\beta}_{\epsilon}^i.
\end{equation*}
Then by the conservation of mass and energy, we obtain that
\begin{equation}\label{e5.2}
\lim_{n\to +\infty}\|\psi_n(t,\cdot)\|_2=a,\ \ \lim_{n\to
+\infty}E_{\epsilon}(\psi_n(t,\cdot))={\beta}_{\epsilon}^i,\
\text{for\ any\ }t\geq 0.
\end{equation}

Define
\begin{equation*}
\varphi_n(t,\cdot)=\frac{a\psi_n(t,\cdot)}{\|\psi_n(t,\cdot)\|_2},\
t\geq 0.
\end{equation*}
Then $\varphi_n(t,\cdot)\in S(a)$ and
\begin{equation}\label{e4.21}
\|\varphi_n(t,\cdot)-\psi_n(t,\cdot)\|\to 0\ \text{as}\ n\to+\infty
\ \text{uniformly\ in\ }t\geq 0,
\end{equation}
which combined with (\ref{e4.19}) gives that
\begin{equation*}
\lim_{n\to+\infty}\|\varphi_n(0,\cdot)-v\|=\lim_{n\to+\infty}
\left\|\frac{au_{0,n}}{\|u_{0,n}\|_2}-v\right\|=0.
\end{equation*}
Hence,  $\varphi_n(0,\cdot)\in \theta_{\epsilon}^{i}\backslash
\partial\theta_{\epsilon}^{i}$ for $n$ large enough because $v\in \theta_{\epsilon}^{i}\backslash
\partial\theta_{\epsilon}^{i}$. Using the method of continuity, $\lim_{n\to
+\infty}E_{\epsilon}(\varphi_n(t,\cdot))={\beta}_{\epsilon}^i$ for
all $t\geq 0$, and
${\beta}_{\epsilon}^i<\tilde{\beta}_{\epsilon}^i$, we obtain that
\begin{equation}\label{e5.1}
\text{for}\ n\ \text{large\ enough},\ \varphi_n(t,\cdot)\in
\theta_{\epsilon}^i\backslash\partial\theta_{\epsilon}^i\ \text{for\
all\ }t\geq 0.
\end{equation}
From (\ref{e5.2})-(\ref{e5.1}), $\{\varphi_n(t_n,\cdot)\}\subset
\theta_{\epsilon}^i$ is a minimizing sequence of $E_{\epsilon,T}$ at
level $\beta_{\epsilon}^i$, and from the proof of Theorem
\ref{thm1.1}, there exists $\tilde{v}\in \theta_{\epsilon}^i$ such
that
\begin{equation}\label{e33.25}
\lim_{n\to +\infty}\|\varphi_n(t_n,\cdot)-\tilde{v}\|=0,
\end{equation}
which combined with (\ref{e4.21}) gives that
\begin{equation*}
\lim_{n\to +\infty}\|\psi_n(t_n,\cdot)-\tilde{v}\|=0.
\end{equation*}
That contradicts (\ref{e44.2}). Hence $\Omega_i$ is orbitally stable
for any $i=1,2,\cdots,l$.
\end{proof}

\begin{theorem}\label{thm5.2}
Let $N,a,\eta,q,h,\epsilon_0$ be as in Theorem \ref{thm1.1},
$p=2^*$, $\epsilon\in(0,\epsilon_0)$, (\ref{e2.1}) and (\ref{e2.4})
hold. Further assume that $h(x)\in C^1(\mathbb{R}^N)$ and $h'(x)\in
L^{\infty}(\mathbb{R}^N)$. Then $\Omega_i \, (i=1,\cdots,l)$ is
orbitally stable under the flow associated with the problem
(\ref{e4.17}).
\end{theorem}

\begin{proof}
The proof can be done by modifying the arguments of Sections 3 and 4
in \cite{JEANJEAN-JENDREJ} and using the arguments of the proof of
Theorem \ref{thm5.3}, so we omit it.
\end{proof}

\bigskip

\textbf{Acknowledgements.}  This work is supported by the National
Natural Science Foundation of China (No. 12001403).



\end{document}